\documentclass[11pt]{elsarticle}
\usepackage[paper=letterpaper,centering,margin=1.5in]{geometry}

\usepackage[utf8]{inputenc}
\usepackage{enumitem, hyperref}
\usepackage{amsmath, amsthm, amssymb}
\usepackage{bm}
\usepackage{graphicx}
\usepackage{epsfig}
\usepackage{comment}
\usepackage{mathenv}
\usepackage{tikz,xcolor}

\usepackage{caption}
\usepackage{subcaption}

\newtheorem{theorem}{Theorem}

\newtheorem{definition}{Definition}

\newtheorem{proposition}{Proposition}
\newtheorem{remark}{Remark}

\newcommand\proba[1]{\mathbb{P}\left(#1\right)}
\newcommand\norm[1]{\left\|#1\right\|}
\newcommand\abs[1]{\left|#1\right|}
\newcommand\expect[1]{\mathbb{E}\left[#1\right]}

\newcommand\bbm{\mathbf{m}}
\newcommand\bbx{\mathbf{x}}
\newcommand\bbz{\mathbf{z}}
\newcommand\bX{\mathbf{X}}
\newcommand\calF{\mathcal{F}}
\newcommand\calM{\mathcal{M}}
\newcommand\calE{\mathcal{E}}
\newcommand\calA{\mathcal{A}}
\newcommand\calP{\mathcal{P}}

\newcommand\strat{\mathcal{S}}
\newcommand\Z{\mathbb{Z}}
\newcommand\R{\mathbb{R}}
\newcommand\N{\mathbb{N}}
\DeclareMathOperator*{\argmin}{arg\,min}

\newcommand\cro[1]{\left\langle#1\right\rangle}

\newcommand*{\id}{{\normalfont\hbox{1\kern-0.15em \vrule width .8pt depth-.5pt}}}

\def\namedlabel#1#2{\begingroup
    #2%
    \def\@currentlabel{#2}%
    \phantomsection\label{#1}\endgroup
}

\title{
Discrete Mean Field Games: Existence of Equilibria and Convergence}

\author[ehu]{Josu Doncel}
\ead{josu.doncel@ehu.eus}

\author[uga,inria]{Nicolas Gast}
\ead{nicolas.gast@inria.fr}

\author[uga,inria]{Bruno Gaujal}
\ead{bruno.gaujal@inria.fr}

\address[ehu]{University of the Basque Country UPV/EHU, Spain}
\address[uga]{Univ. Grenoble Alpes, F-38000 Grenoble, France}
\address[inria]{Inria}

\date{}

\begin{document}

\begin{abstract}
  We consider  mean field games with  discrete state spaces (called  discrete mean field games in the following) and we analyze these games in continuous and discrete time, over finite as
  well as infinite time horizons. We prove the existence of a
  mean field equilibrium assuming continuity of the cost and of the drift. These conditions
  are more general than the existing papers studying finite 
  state space mean field games. Besides, we also study
  the convergence of the equilibria of $N$-player games to mean field
  equilibria in our four settings. On the one hand, we define a class  of strategies in which any
  sequence of equilibria  of the finite games converges weakly to a mean field equilibrium when the number of
  players goes to infinity.  On the other hand, we exhibit equilibria outside this
  class that do not converge to mean field equilibria and for which
  the value of the game does not converge.
  In discrete time this non-convergence phenomenon implies that
  the Folk theorem does not scale to the mean field limit.
\end{abstract}

\maketitle

%
%

\section{Introduction}

Mean field games have been introduced by Lasry and Lions \cite{LL07}
as well as Huang, Caines and Malhamé~\cite{Caines} to model
interactions between a large number of strategic agents (players) and
have had a large success ever since.  Since the seminal work in
\cite{LL06,LL06-2,LL07,Caines}, a large variety of papers have been
investigating mean field games. Most of the literature concerns
continuous state spaces and describes a mean field game as a coupling
between a Hamilton-Jacobi-Bellman equation with a Fokker-Planck
equation (see for example
\cite{OLL11,bensoussan2013mean,cardaliaguet2015master,gomes2015regularity,carmona2013probabilistic,gomes2015time,gomes2015time2,gomes2016time,AMBROSE2018141}). Here,
we are interested in studying mean field games with a finite number of
states and finite number of actions per player. In this case, the
analog of the Hamilton-Jacobi-Bellman equation is the {\it Bellman
  equation} and the discrete version of the Fokker-Planck equation is
the {\it Kolmogorov equation}.

Finite state space mean field games in {\it discrete time} ({\it
  a.k.a.} with synchronous players) were previously studied in
\cite{Gomes2010308}. In their work, the strategy of the players is the
probability matrix of the Kolmogorov equation. This implies that each
player can choose her dynamics independently of the state of the
others: the behavior of players is only coupled via their costs. In
that case, the Kolmogorov equation becomes linear.

Finite state space mean field games in {\it continuous time} ({\it
  a.k.a.} with asynchronous players) have also been previously
analyzed in \cite{GMS13,Gueant2014,bayraktar2017analysis,carmona2016finite}. In their
model, the players also control completely the transition rate matrix
so that the dynamics are again linear once the actions of the players
are given. Again, players do not interact with each other directly in
these models, but only through their costs.

The models we study here, both in the synchronous and
asynchronous cases  cover non-linear dynamics:  We  consider  that the players do not
have the power to choose the rate matrix and that their actions only
have a limited effect on their state.  Here, the transition rate matrix
may depend not only on the actions taken by the player, but also on
the population distribution of the system. This introduces an explicit
interaction between the players (and not just through their
costs). This non-linear dynamics is called {\it the relaxed case} in \cite{Fischer}. We claim that the model with explicit interactions covers
several natural phenomena such as information/infection propagation or
resource congestion where the cost but also the state dynamics of a
player depend on the state of the all the others. This type of
behavior is classical in systems with a large number of interacting
objects \cite{benaim2008class} and cannot be handled using previous
mean field game models. For instance, in the classical SIR
(Susceptible, Infected, Recovered) infection model
\cite{wang2016statistical}, the rate of infection of one individual
depends on the proportion of individuals already infected.  Similarly,
in a model of congestion one player cannot typically use a resource if
it is already used to full capacity.

We show that the only requirement needed to guarantee
the existence of a Mean Field Equilibrium in mixed strategies is that the cost is
continuous with respect to the population distribution (convexity is
not needed).  This result nicely mimics the conditions for existence
of a Nash equilibrium in the simpler case of static population games
(see \cite{Sandholm}). The existence of a mean field equilibrium
in mixed strategies has been previously shown by \cite{Lacker2016,carmona2015probabilistic}
in the diffusion case. 
In \cite{Gueant2014} the existence of a Mean
Field Equilibrium is proven under the assumption that the cost of a
player is strictly convex w.r.t. her strategy and in \cite{GMS13} the
authors also consider uniformly convex functions. These conditions are
rather strong because they are not satisfied in the important case of
linear and/or expected costs.
In \cite{Fischer} existence of a Nash equilibrium is also proved under mere continuity assumptions
and with a compact action space (more general than the simplex, used here).
However, the main  difference between the two approaches  is the type of mean field limit that is used. In \cite{Fischer},   the trajectories  of the states of the players are considered while we only consider the state at time $t$. The first approach uses arguments in line with the propagation of chaos while the second one is closer to the work in \cite{Kolokotsov,tembine2009mean}.
    While the convergence of trajectories is a more refined convergence than the point-wise convergence in general, this is useless here. Indeed, for mean field games,  costs are associated to states and actions and not to  trajectories. Therefore, the point-wise mean field approach is sufficient.
Another difference with \cite{Fischer} is that  an additional assumption about the uniqueness of the argmin is needed in some parts of the convergence proof as well as for existence  (in the feedback case). This is not the case here, so both papers do not cover the exact same set of games.

As in most existence proofs, our  proof is based on a version of the fixed point theorem of Kakutani
in infinite dimension (see for example  \cite{carmona2016finite}
where such an extended version of the fixed point theorem is used in a mean field game model with minor and major players). Here however, we do not  consider the
best response operator but the evolution of the population
distribution instead, as in \cite{Fischer}.  Out of the four cases (asynchronous/synchronous, finite/infinite horizons),
we mainly detail the asynchronous player case
for which we prove this existence of a mean field equilibrium in an
infinite horizon with discounted costs. We also show, more briefly, how these results
can be extended to a finite horizon or to a finite or infinite time
horizon in the synchronous-player case.

Our second contribution concerns convergence of finite games to mean
field limits. 
  Different authors have studied the convergence of
$N$-player games equilibria to mean field equilibria, \emph{e.g.}
\cite{huang2012mean,adlakha2013equilibria,tembine2011mean,tembine2009mean}.
The type of strategies considered in these paper is different from
ours: they consider that the strategy of a player only depends on her
internal state (these are called {\it stationary policies} in
\cite{tembine2009mean}), whereas here we allow time dependence in these
policies. The model in \cite{tembine2009mean} does include state
dynamics that depend on the population distribution but only considers
stationary strategies that do not depend on time, hence cannot depend
on the population dynamics.

In all four combinations (finite / infinite horizon, synchronous /
asynchronous), a mean field equilibrium is always an
$\epsilon$-approximation of an equilibrium of a corresponding game
with a finite number $N$ of players, where $\epsilon$ goes to 0 when
$N$ goes to infinity. This is the discrete pending result to similar results in continuous games \cite{Delarue}.
However, we show also that not all equilibria for
the finite version converge to a Nash equilibrium of the mean field
limit of the game.  We provide several counter-examples to illustrate
this fact.  They are all based on the following idea: The ``tit for
tat'' principle allows one to define many equilibria in repeated games
with $N$ players. However, when the number of players is infinite, the
deviation of a single player is not visible by the population that
cannot punish him in retaliation for her deviation. This implies that
while the games with $N$ players may have many equilibria, as stated
by the folk theorem, this may not be the case for the limit game.
This fact is well-known for large repeated games (see examples of  {\it Anti-folk Theorems} in \cite{Sabourian,Nabil}).
However, up to our knowledge, these results  have not yet been investigated in the mean field game framework.
\footnote{An extended
  abstract discussing our  counterexample  in  the continuous time model
  with infinite horizon was presented in \cite{DGG16}.}

Finally, our four models of dynamic games do not face the issue of the order of play, nor partial information.
Thus, we avoid two
difficulties of dynamic games: the information structure of each
player and the existence of a value \cite{Dasgupta}.  In our case, all
players are similar, so the order of play is irrelevant, and we only
consider the full information case: players know the strategy of the
other players and the current global state (more details on this are given in Section~\ref{ssec:info}). 

The rest of the article is organized as follows.  We introduce mean
field games with explicit interactions in continuous time in
Section~\ref{sec:continuous-time} where we mainly focus on the
infinite horizon with discounted costs.  We describe the evolution of
the state of the players, the cost function as well as the
best response operator. In both cases (finite and infinite horizon),
we prove the existence of an equilibrium.  We show in Section
\ref{sec:conv} that this equilibrium is an approximation of an
equilibrium for the game with a finite number of players.  Finally, we
study an example of an $N$-player game inspired from the prisoner's
dilemma whose equilibria are not always equilibria for the limit mean
field game.  We focus on the synchronous case in
Section~\ref{sec:discrete-time} (where players all play at the same
time).  In this case, $N$-player games can be seen as classical
stochastic games in discrete time.  We derive the mean field limit
dynamics and the existence of an equilibrium.  Here counter-examples
of equilibria for finite games that do not go to the limit are easier
to find. Indeed, the folk theorem applies and all equilibria based on
retaliation cannot be equilibria at the limit.


%
%

\section{Discrete Mean Field Games in Continuous Time}
\label{sec:continuous-time}

\subsection{Notations and Definitions}
\label{sec:model:sub:notations-defs}
\label{sec:def}

A discrete mean field game ${\cal G}$ is a  tuple
${\cal G} = (\calE,\calA,\{Q_a\},\bbm_0,\{c_a\},\beta)$, where $\calE$ is
the state space, $\calA$ the action set, $\{Q_a\}$ the transition rate
matrices, $\bbm_0$ the initial state, $\{c_a\}$ the cost functions and
$\beta \in \R$ a discount factor.

The game is described as follows.

\paragraph{\bf State and action sets}
We consider a {\it population} made of an infinite number of  homogeneous {\it players} that evolve in continuous time.  Each player has a finite state space denoted by
$\calE=\{1,\dots,E\}$ and a finite action set
$\mathcal{A}=\{1,\dots,A\}$. 

We denote by $\calP(\calA)$ (resp.  $\calP(\calE)$) the set of probability measures over $\calA$ (resp.  $\calE$). Since $\calA$ is finite,
$\calP(\calA)$ is the simplex of dimension $A$.

\paragraph{\bf Set of strategies}
A {\it mixed strategy} (or strategy for short) is a measurable
function $\pi:\calE\times\R^+\to\calP(\calA)$, that associates to each
state $i\in\calE$ and each time $t\ge0$ a probability measure
$\pi_i(t) \in \calP(\calA)$ on the set of possible actions.  We also
denote by $\pi_{i,a}(t)$ the probability that, at time $t$, a player
in state $i$ takes the action $a$, under strategy $\pi$. For all
$t\ge0$ and all $i\in\calE$, we have $\sum_{a\in\calA}\pi_{i,a}(t)=1$.
The set of all possible strategies is denoted by $\strat$.

We say that a strategy is
{\it pure}  if, for all state $i$ and all $t\in\mathbb{R}$, there
exists an action $a\in\calA$ such that $\pi_{i,a}(t)=1$ and
$\pi_{i,a^\prime}(t)=0$ for all $a^\prime\neq a$.

The set $\strat$ is a bounded subset of the Hilbert space of the
functions $\calE\times\R^+\to\R^A$ equipped with the inner product the
exponentially weighted inner product :
$\cro{f,g}=\int_0^\infty f(g)g(t)e^{-\beta t}dt$. This shows that
$\strat$ is weakly compact, where the weak topology is defined as
follows: a sequence of policy $\pi^n$ converges to a policy
$\pi$ if for any bounded function $g$:
\begin{align*}
  \lim_{n\to\infty}\int_{0}^\infty \pi^n(t)g(t)e^{-\beta t}dt = \int_0^\infty
  \pi(t)g(t)e^{-\beta t}dt. 
\end{align*}

\paragraph{\bf Rate matrices}

We denote by $\bbm^\pi(t)\in\calP(\calE)$ the {\it population
distribution} at time $t$.  As the state space is finite,
$\bbm^{\pi}(t)$ is a vector whose $i$-th component,
$m^{\pi}_i(t)$, is the proportion of players in state $i$ at time $t$.
The evolution over time of the population distribution is driven by rate matrices:
 $\{ Q_a(\bbm^{\pi}(t))\}_{a\in \calA}$. 
By definition, $Q_{ija}(\bbm^{\pi}(t))$ is the rate at which a player in state $i$ moves to state $j$ when choosing action $a$, when the population distribution is $\bbm(t)$. Note that by definition, $\sum_{j\in\calE}Q_{ija}(\bbm^{\pi}(t))=0$ for all $i$ and $a$
and $Q_{ija}(\bbm^{\pi}(t))$ is non-negative for all $j\neq i$ and all $a$.

In the following, we assume that for all $i,j,a$,  $Q_{ija}(\bbm)$ is Lipschitz-continuous in $\bbm$ with constant $L$.

The initial condition is $\bbm^{\pi}(0)=\bbm_0$.   For $t\ge0$, the population distribution $\bbm^\pi(t)$ is
the solution of the following differential equation, that depends on
the strategy $\pi$: for $j\in\calE$
\begin{equation}
  \dot{m}^{\pi}_j(t)= \sum_{i\in\calE}\sum_{a\in\calA}{m^{\pi}_i(t) Q_{ija}(\mathbf{m}^{\pi}(t)) \pi_{i,a}(t)} .
  \label{eq:emp-measure}
\end{equation}
The rationale behind this differential equation is that all players in state $i$ use  the action
$a\in\mathcal{A}$ and  move to state $j$
with rate $Q_{ija}(\bbm^{\pi}(t))$. 

If the strategy $\pi_i(t)$ is not continuous in time, the differential
equation \eqref{eq:emp-measure} may not be well-posed at time-points
where $\pi_i$ is not continuous.  The existence of a continuous solution
for \eqref{eq:emp-measure} is guaranteed by the Carathéodory's Existence
Theorem.  The Lipschitz condition on $Q$ further implies that this solution is
essentially unique because any solution of \eqref{eq:emp-measure} must be a fixed point of 
\begin{equation}
  {m}^{\pi}_j(t)= m_{j,0} + \int_0^t \left( \sum_{i\in\calE}\sum_{a\in\calA}{m^{\pi}_i(u) Q_{ija}(\mathbf{m}^{\pi}(u)) \pi_{i,a}(u)} \right)du .
  \label{eq:emp-measure2}
\end{equation}

In anticipation, the same properties (existence and uniqueness of the
solution of the ODE) hold for the differential equation
\eqref{eq:proba-state}.

\begin{remark}[Explicit interactions]
In this model, the rate matrix  $Q_{ija}(\bbm^{\pi}(t))$ depends
explicitly on the population distribution: the rate to go from state
$i$ to state $j$ under action $a$ depends on how the whole population is
distributed among the states of the system.
Other mean field models, such as \cite{Gomes2010308}, only consider the
special case
where  $Q_{ija}(\bbm^{\pi}(t))$ is constant:   $Q_{ija}(\bbm^{\pi}(t))
= Q_{ija}$. This  restricts the population dynamics given in
\eqref{eq:emp-measure} to linear dynamics.
\end{remark}

\paragraph{\bf Cost function}
We now concentrate on a particular player, that we call Player~0. 
Player~0 chooses her own strategy
$\pi^0:\R^+\times\calE\to\calP(\calA)$.
We denote by $\mathbf{x}^{\pi^0}(t)\in\calP(\calE)$ the
probability distribution of Player~0 when Player~0 uses strategy
$\pi^0$ against a population who has distribution $\bbm$. 
For a given state $i\in\calE$,
$x^{\pi^0,\bbm}_i(t)$ denotes the probability for Player~0 to be in
state $i$ at time $t$.  The distribution $\mathbf{x}^{\pi^0,\bbm}$
evolves over time according to the following differential equation:
for $j\in\calE$
\begin{equation}
  \dot{x}^{\pi^0,\bbm}_j(t)=\sum_{i\in\calE}\sum_{a\in\calA}{x_i^{\pi^0}(t) Q_{ija}(\bbm(t)) \pi^0_{i,a}(t)}.
\label{eq:proba-state}
\end{equation}
If Player~0 is in state $i$ and takes an action $a$, it suffers from
an {\it instantaneous cost} $c_{i,a}(\bbm(t))$, that depends on
the population distribution at time $t$. 
We assume that the cost is always continuous in $\bbm$. Given a
population distribution $\bbm$ and the strategy of Player~0 $\pi^0$,
we define the discounted 
cost of Player~0 as
\begin{equation}
  \label{eq:V(pi0,pi)}
  W(\pi^0,\bbm)=\int_0^\infty \left(\sum_{i\in\calE}\sum_{a\in\calA}{x_i^{\pi^0,\bbm}(t) c_{i,a}(\bbm(t)) \pi^0_{i,a}(t) e^{-\beta t}}\right)\ dt,
\end{equation}
where $\beta>0$ is the discount factor.

We also introduce the notation $V(\pi^0,\pi)$ that represents the
discounted cost of Player~0 when the population plays a strategy
$\pi$:
\begin{align*}
  V(\pi^0,\pi) = W(\pi^0,\bbm^{\pi}). 
\end{align*}

\paragraph{\bf Best response}
The {\it best response } to $\pi$ of Player~0 is to choose a strategy $\pi^0 \in \strat$ that minimizes
her discounted cost~\eqref{eq:V(pi0,pi)} when the rest of the
population plays strategy $\pi$.  For a given population strategy
$\pi$, we denote the set of best responses of Player~0 to $\pi$ by
$BR(\pi)$.  This set is the set of strategies that minimizes her
discounted cost:
\begin{equation}
  BR(\pi):=\argmin_{\pi^0\in\strat}V(\pi^0,\pi).
\label{eq:BR}
\end{equation}

Note that the best response function is well defined or, in other
words, that the ``argmin'' is reached for some strategy in Equation
\eqref{eq:BR}. To prove that, we will later prove in
Section~\ref{sec:proofs} that the function $V$ is continuous for the
weak topology. As $\strat$ is weakly compact, this shows that the
minimum in $\pi^0$ is attained.
\begin{proposition}
  \label{prop:V}
  The function $V$, defined in Equation~\eqref{eq:V(pi0,pi)} is
  continuous in $\pi^0$ and $\pi$ (for the weak-topology on
  $\strat$). 
\end{proposition}

\paragraph{\bf  Mean field equilibrium}
We then define a mean field equilibrium as a strategy $\pi^{MFE}$ such
that when the population strategy is $\pi^{MFE}$, a selfish Player~0
would also choose the same strategy $\pi^{MFE}$ as her best response.

\begin{definition}[Mean Field Equilibrium]
  A strategy $\pi$ is called a mean field equilibrium if it is a fixed
  point for the best response function, i.e.,
  \begin{equation}
    \pi^{MFE}\in BR(\pi^{MFE}).\label{eq:mfe}
  \end{equation}
  A mean field equilibrium is  pure if it is a pure
  strategy.
\end{definition}

The rationale behind this definition is when one considers that the
population is formed by players that each take selfish
decisions. As the population is homogeneous, each player best response
is the same as Player~0. In other words, for a given
population strategy $\pi$, all the rational players of the populations
(or players) choose the strategy $BR(\pi)$.  As in classical games, a mean field equilibrium
is a situation where no player has incentive to deviate unilaterally
from the common strategy.

%
%

\subsection{Existence of Mean Field Equilibrium}
\label{sec:existence-equil}

We now show that, under very general assumptions,  all discrete mean field games admit  a
mean field equilibrium.  As for classical games, these equilibria are
not necessarily \emph{pure}.  As most proof on existence of equilibria, our proof relies on a generalization of Kakutani fixed
point theorem to infinite dimensional spaces. However, the classical approach consisting of showing that  the best response function BR($\pi$) is a
Kakutani map  does not work here when  the cost function is not 
strictly convex.
  Therefore, in our approach we focus on the  state of the game
  instead of the best response function.

As mentioned before, the differential
equations~\eqref{eq:emp-measure}, \eqref{eq:proba-state} and the cost
equation \eqref{eq:V(pi0,pi)} are all well defined under our running 
Assumption~(A1): 
\begin{itemize}
\item[(A1)] The rate function $\bbm\mapsto Q_{ija}(\bbm)$ is
  Lipschitz-continuous in $\bbm$. The cost function
  $\bbm\mapsto c_{i,a}(\bbm)$ is continuous in
  $\bbm$.
\end{itemize}
In particular, this assumption implies that 
the costs and the rates are all bounded by a finite value.

\begin{theorem}
  Any  discrete  mean field game ${\cal G}$ whose rate and cost 
  satisfy Assumption~(A1) admits a mean field equilibrium.
  \label{thm:mfe-existence}
\end{theorem}

Note that in general, the best response function $\pi\mapsto BR(\pi)$
is neither continuous nor hemi-continuous in general under (A1). In
particular, the best response set $BR(\pi)$ may not be a convex set.
This makes difficult the application of the classical fixed point
theorems on the best response function.  As a result, our proof will
formulate the fixed point problem in an alternative manner by
considering a fixed point in $\bbm$.

\subsection{Proofs}
\label{sec:proofs}

\subsubsection{Proof of Proposition~\ref{prop:V}}

For a strategy $\pi$, the function $\bbm^{\pi}$ satisfies the
differential equation~\eqref{eq:emp-measure}. As $\bbm^\pi(t)$ lives
in a compact and the functions $Q$ are continuous, the right-hand side
of this differential equation is bounded. This shows that there exists
a constant $L'$ such that for any strategy $\pi$, the function
$\bbm^{\pi}$ is Lipschitz-continuous with constant $L'$. Similarly the
function $\bbx^{\pi^0}$ is also Lipschitz-continuous with constant
$L'$.

Let $\calM$ be the set of functions from $\R^+$ to $\calP(\calE)$ that
are Lipschitz-continuous with constant $L'$. We equip this set with
the exponentially weighted $L_\infty$-norm :
\begin{align*}
  \norm{\bbm-\bbm'}=\sup_{i\in\calE,t\ge0}\abs{m_i(t)-m'_i(t)}e^{-\beta
  t}. 
\end{align*}
By the Arzela-Ascoli theorem, $\calM$ is a compact space. 
 
To prove that $V$ is continuous in $\pi$ and $\pi^0$, it suffices to
show that the mapping $\pi\mapsto\bbm^{\pi}$ is continuous (for the
weak topology) and that the mapping
$(\pi^0,\bbm)\mapsto x^{\pi^0,\bbm}$ is continuous.  To prove the
continuity of $\bbm^{\pi}$, let $\pi_n$ be a sequence of strategy that
converges to a strategy $\pi$. As $\calM$ is compact, there exists a
function $\bbm$ and a subsequence of $m^{\pi_n}$ that converges to
$\bbm$. Moreover, we have :
\begin{align}
  {m}_j(t) &= m_{j,0} + \lim_{n\to\infty}\int_0^t \left(
             \sum_{i\in\calE}\sum_{a\in\calA}{m^{\pi_n}_i(u)
             Q_{ija}(\mathbf{m}^{\pi^{n}}(u)) (\pi_{n})_{i,a}(u)}
             \right)du\nonumber\\
           &= m_{j,0} + \int_0^t \left(
             \sum_{i\in\calE}\sum_{a\in\calA}{m_i(u)
             Q_{ija}(\mathbf{m}(u)) \pi_{i,a}(u)}
             \right)du,\label{eq:V_proof} 
\end{align}
where the convergence holds because $\pi_n$ converges weakly to $\pi$
and $\bbm^{\pi_n}$ converges uniformly on all compact to $\bbm$. 

Equation~\eqref{eq:V_proof} shows that the function $\bbm$ is equal to
the function $\bbm^{\pi}$. This shows that $\pi\to\bbm^{\pi}$ is
continuous in $\pi$ which implies that $V$ is continuous in $\pi$.

The proof that $(\pi^0,\bbm)\mapsto x^{\pi^0,\bbm}$ is continuous is
very similar to the above proof and we therefore omit it.

\subsubsection{Proof of Theorem~\ref{thm:mfe-existence}}

Recall that for a given population distribution $\bbm\in\calM$, the
cost of a strategy $\pi^0$ is defined as
\begin{align}
  \label{eq:W1}
  W(\pi^0,\bbm) &= \int_0^{\infty}{\left(\sum_{i,a}x_i(t) \pi^0_{i,a}(t)
                  c_{i,a}(\bbm(t)) e^{-\beta t}\right)dt},\\
                &\text{ where $\bbx$ satisfies (for all $j\in\calE$): }\dot{x}_j(t)=\sum_{i,a}{x_i(t)
                  Q_{ija}(\bbm(t))\pi^0_{i,a}(t)}.
    \label{eq:W2}
\end{align}

We now define the function $\Phi:\calM\to 2^\calM$ as the best
response to a population distribution $\bbm$. It is a mapping that
associates to a population distribution $\bbm\in\calM$, the set of all
state distributions that can be induced by an optimal policy:
\begin{align}
  \label{def:Phi}
  \Phi(\bbm)=&\left\{\bbx^{\pi^0} \text{ such that }
               \pi^0\in\argmin_{\pi\in\strat}{W(\pi,\bbm)}\right\}.
\end{align}

In the remainder of the proof, for all $\bbm\in\calM$, $\Phi(\bbm)$ is
well defined and non empty (\emph{i.e.}, the minimum is attained), is
convex and compact.  Moreover, we will also show that the function
$\Phi(\cdot)$ is upper-semicontinuous. As $\calM$ is compact
\cite[Prop. 11.11]{border1989fixed}, this shows that $\Phi(\cdot)$
satisfies the conditions of the fixed point theorem given in
\cite[Theorem 8.6]{granas2013fixed} and therefore has a fixed point
$\bbm^*$. By the definition of $\Phi$, this implies that there exists
a strategy $\pi^0$ that is a best-response to $m^{\pi^0}$, which
implies that $\pi^0$ is a mean field equilibrium.

\textbf{Definition of $\Phi(\bbm)$} -- It can be shown that $W$ is
continuous (by using a reasoning similar to the one for $V$
(Proposition~\ref{prop:V})).  This shows that there exists  $\pi^0$
that attains the minimum on the right hand side of
Equation~\eqref{def:Phi}, which shows that $\Phi(\bbm)$ is well
defined and non-empty.

\textbf{Compactness of $\Phi(\bbm)$} -- Let us consider the following
optimization problem:
\begin{align}
  \label{eq:Z1}
  &\min_{\bbx,\bbz}
    \int_0^{\infty}{\left(\sum_{i,a}z_{i,a}(t) c_{i,a}(\bbm(t)) e^{-\beta t}\right)dt}\\
  & \text{such that $\bbz$ satisfies }
    \left\{
    \begin{array}{ll}
      \sum_{a}z_{j,a}(t)=x_j(t) &\forall j\in\calE,\\
      z_{j,a}(t)\geq 0, &\forall j\in\calE, \forall a\in\mathcal{A},\\
      \dot{x}_j(t)=\sum_{i,a}{z_{i,a}(t)Q_{ija}(\bbm(t))} &\forall j\in\calE.
    \end{array}
                                                            \right.
                                                            \label{eq:Z2}
\end{align}
The above problem is a linear problem, which implies that the set of
optimal solutions is convex and compact. Let us show that the set of
optimal solution of the optimization problem~\eqref{eq:Z1} is
$\Phi(\bbm)$. To show this, let us remark that  the constraints \eqref{eq:W2} are
equivalent to the constraints \eqref{eq:Z2} by replacing the variables
$x_i(t) \pi^0_{i,a}(t)$ by $z_{i,a}(t)$. Then, the constraint
$\pi\in\strat$ of \eqref{eq:W2}, that corresponds to
$\pi^0(t)\in\calP(\calA)$, is replaced with $z_{i,a}(t)\ge0$ and
$\sum_a z_{i,a}(t)=x_i(t)$.

\textbf{Upper-semi continuity of $\Phi$}. To prove that $\Phi$ is
upper-semi continuous, let us show that the graph of
$\bbm\mapsto\Phi(\bbm)$ is closed. Let $\bbm_n\in\calM$ and
$\bbx_n\in \Phi(\bbm_n)$ be two sequences such that
$\lim_{n\to\infty} \bbm_n= \bbm_\infty$ and
$\lim_{n\to\infty} \bbx_n= \bbx_\infty$. We want to show that
$\bbx_\infty\in \Phi(\bbm_\infty)$.

As $W$ is continuous, for all $\bbx_n\in\Phi(\bbm_n)$, there exists a
strategy $\pi_n$ that minimizes $W(\pi,\bbm_n)$ and such that
$\bbx_n=\bbx^{\pi_n,\bbm_n}$. As the set $\strat$ is weakly compact,
this sequence of strategies has a subsequence that converges weakly to
a strategy $\pi_*$. Moreover, we have:
\begin{itemize}
\item As $W$ is continuous, $\pi_*$ minimizes
  $W(\pi,\bbm_\infty)$. This shows that
  $\bbx^{\pi_*}\in\Phi(\bbm_\infty)$.
\item The solution of \eqref{eq:W2} is continuous in $\pi$ and $\bbm$,
  which shows that $\bbx_\infty=\bbx^{\pi_*,\bbm_\infty}$.
\end{itemize}
Combining these two facts shows that $\bbx_\infty\in\Phi(\bbm_\infty)$
which implies that the graph of $\Phi$ is closed. 
\qed

\begin{remark}
  The continuity assumption (A1) is tight in the following sense: 

1- If the rate $Q$ is not Lipschitz-continuous in $\bbm$, then the evolution of the population is not well defined, in the sense that the evolution equation \eqref{eq:emp-measure} may have several solutions or no solution at all.

2- There exist games with non-continuous cost functions that do not
admit any mean field equilibrium. For example, consider the following
mean field game:
\begin{align*}
  {\cal G} = \bigg(\calE = \{1,2\}, \calA = \{a,b\}, Q_a = 0, Q_b = \left[
  \begin{array}{cc}
    -1&1\\
    0&0
  \end{array}
       \right],m(0)=(1,0)
  \\
  c_a(m_1,m_2) = 0, c_b(m_1,m_2) = \left\{
  \begin{array}{cc}
    -1&\text{if $m_2\le1/2$}\\
    1&\text{otherwise}
    \end{array}
        \right.,
        \beta\bigg). 
\end{align*}
Assume that this game has a mean field equilibrium and let denote by
$m(t)$ the state at equilibrium. By definition of $Q_a$ and $Q_b$,
$m_2(t)$ is a non-decreasing function. Hence, let
$\tau=\sup\{t: m_2(t)\le1/2\}$ (note that
$\tau\in[\ln2;+\infty)\cup\{+\infty\}$). It should be clear that the
best response of Player~0 to any state function $m$ is the policy
$\pi^{(\tau)}$ that consists in playing ``$b$'' until $\tau$ and
``$a$'' after $\tau$. However, such a policy is never a mean field
equilibrium: under the policy $\pi^{(\tau)}$,
$m_2(t)=1-e^{-\min(t,\tau)}$, which means that
$\sup\{t: m_2(t)\le1/2\}\in\{\ln2,+\infty\}$. None of the policies
$\pi^{(\ln2)}$ or $\pi^{(\infty)}$ is an equilibrium: the policy
$\pi^{(\ln2)}$ is the best response to $\pi^{(\infty)}$ and
vice-versa.

\end{remark}


\section{Convergence of Finite Games to Mean Field Games}
\label{sec:conv}

Mean field games are often presented as a limit of a sequence
of finite games as the number $N$ of players goes to infinity. In this
section, we investigate positive and negative results that link finite
games and mean field games. 

\subsection{Markov Game with $N$ Exchangeable Players}
\label{sec:continuous_N}

To any discrete mean field game
${\cal G}= (\calE,\calA,\{Q_a\},\bbm_0,\{c_a\},\beta)$, one can
associate a stochastic $N$-player game
${\cal G}^N= (N,\calE,\calA,\{Q_a\},\bbm_0,\{c_a\},\beta)$ as
follows. The finite stochastic game ${\cal G}^N$ has the same state
and action spaces $\calE,\calA$, the same rate matrices $Q_a$, the
same cost functions $c_a$, the same discount factor $\beta$, and the
same initial state as ${\cal G}$.  The time evolution of the finite
game is as follows.  At any time $t$, each player (say Player~$n$)
chooses a (randomized) action $A_n(t) \in \calP(\calA)$.

We consider a mean field interaction model between the players, which
means that the behavior of one object only depends on the states of
the other objects through the proportion of objects that are in a
given state. To be more precise, we denote by
${\bf M}(t)\in\calP(\calE)$ the population distribution of the system
at time $t$. As the set $\calE$ is finite, ${\bf M}(t)$ is a vector
with $|\calE|$ components and for all $i\in\calE$, $M_i(t)$ is the
fraction of players that have state $i$ at time $t$:
\begin{equation*}
  M_i(t) = \frac{1}{N}\sum_{n=1}^N \mathbf{1}_{\{X_n(t)=i\}}. 
\end{equation*}

The state  of one player (say Player~n) follows a continuous time Markov chain whose rate varies over time.
The only dependence between players is through the rate that depends on the population distribution.

More precisely, 
the evolution of the state of Player~$n$, under
$\calF_t$, the natural filtration of the process, satisfies for all $k \in \N$ and all states $i\not= j$,
\begin{equation}
  \proba{X_n(t + dt)=j \middle| X_n(t)=i ,  {\bf M}(t) =
    \bbm , A_n(t)=a , \calF_t} =
   Q_{ija}(\bbm) dt + o(dt),
  \label{eq:R_N_c}
\end{equation}
where $A_n(t)$ is the action taken by Player~$n$ at time $t$.

At any time $t$, Player~$n$ suffers an instantaneous cost that is a
function of her state $X_n(t)$, the action that she takes $A_n(t)$ and
the population distribution ${\bf M}(t)$. We write this instantaneous cost
$c_{X_n(t),A_n(t)}({\bf M}(t))$.

The objective of Player~$n$ is to choose a
strategy $\pi^n$ from some set of admissible strategies $\Pi$, in
order to minimize her expected discounted cost, knowing the
strategies of the others. As before, the discount factor is denoted
by $\beta$.
Given a strategy $\pi^n\in\Pi$ used by Player~$n$ and a strategy
$\pi\in\Pi$ used by all the others, we denote by $V^N(\pi^n,\pi)$ the
expected discounted cost of Player~$n$:
\begin{equation*}
  V^N(\pi^n,\pi) = \expect{\int e^{-\beta 
      t}c_{X_n(t),A_n(t)}({\bf M}^{\pi}(t))dt \middle| %
    \begin{array}{l}
      \text{$A_n$ is chosen w.r.t. $\pi^n$} \\
      \text{$A_{n'}$ is chosen w.r.t. $\pi$ ($\forall n'\ne n$)}
    \end{array}
  }.
\end{equation*}

A Nash equilibrium for this game is a strategy $\pi$ such that Player $n$
does not have another  admissible strategy that leads to a lower cost.
This notion depends naturally on the set of admissible strategies.
\begin{definition}[Equilibrium of the $N$ player game]
  For a given set of strategies $\Pi$, a strategy $\pi\in\Pi$ is
  called a symmetric equilibrium in $\Pi$ if for any strategy
  $\pi^n\in\Pi$:
  \begin{equation*}
    V^N(\pi,\pi) \le  V^N(\pi^n,\pi). 
  \end{equation*}
\end{definition}

We will also use the notion of $\varepsilon$-equilibrium:
\begin{definition}[$\varepsilon$-equilibrium of the $N$ player game]
  For a given set of strategies $\Pi$, a strategy $\pi\in\Pi$ is
  called an $\varepsilon$- symmetric equilibrium in $\Pi$ if for any strategy $\pi^n\in\Pi$: 
  \begin{equation*}
    V^N(\pi,\pi) \le  V^N(\pi^n,\pi) + \varepsilon. 
  \end{equation*}
\end{definition}

\subsection{Subsets of Admissible Strategies}
\label{ssec:info}

In a \emph{full information} setting, $A_n(t)$ is a (possibly random)
function of the values $X_{n'}(t')$ up to time $t'\le t$ and all
actions taken in the past $A_{n'}(t')$, for $t'<t$ and for
$n'\in\{1\dots N\}$. Such a strategy is, however, hard to
analyze. Therefore, in the following, we will consider two natural
subclasses for the set of admissible strategies, depending on the information available to the players:
\begin{itemize}
\item (Markov) -- A strategy $\pi$ is called a \emph{Markov strategy} if
  it induces a choice of $A_n(t)$ that is a (possibly random) measurable function
  of only $t$, ${\bf M}(t)$ and ${\bf X}(t)$:
  \begin{equation*}
    \proba{A_n(t)=a\mid\calF_t}=\pi_{a,X_n(t)}(t,{\bf M}(t)). 
  \end{equation*}
  This definition is motivated by the fact that, as indicated by
  Equation~\eqref{eq:R_N_c}, the behavior of one object depends on the
  others only through the value ${\bf M}(t)$. This implies that when all the
  other players use a Markov strategy, the set of Markov strategies is
  dominant among the set of full-information strategies: there exists a
  full-information best response for Player~$n$ that is a Markov
  strategy.
Furthermore, any Markov game admits a Markovian Nash equilibrium (see \cite{fink1964}).
\item (Local) -- A strategy $\pi$ is a  \emph{local strategy} if the
  choice of the action only depends on the player's internal state and
  on the time.
  \begin{equation*}
    \proba{A_n(t)=a\mid\calF_t}=\pi_{a,X_n(t)}(t). 
  \end{equation*}
  If a player uses a local  strategy, its actions may depend on time, hence may track the law  of the population  ${\bf M}(t)$ (but not  ${\bf M}(t)$ itself). 
 Also notice  that a local strategy is not necessarily stationary because of its dependence on time.
\end{itemize}

\subsection{Nash Equilibria Limits}

The next theorem provides a relation between local equilibria of
finite games and mean field equilibria of the limit mean field
game. In particular, it shows that mean field equilibria are a good
approximation of local equilibria. However, as we will show later,
this result does not hold for Markovian equilibria.

\begin{theorem}
\label{th:coonv1}
Consider a finite stochastic  game ${\cal G}^N$, with $N$ players and 
  assume that (A1)  holds for its rate matrices $Q_a$ and its cost functions $c_a$. Then:
  \begin{itemize}
  \item[(i)] Let $\pi$ be a mean field equilibrium of the associated mean field game ${\cal G}$. There exists $N_0$ such
    that for all $N\ge N_0$, $\pi$ is a local
    $\varepsilon$-equilibrium of the $N$ player game.
  \item[(ii)] Let  $(\pi^N)_{N\in \N}$ be  a sequence of local strategies such that
    $\pi^N$ is an $\varepsilon_N$-equilibrium for the $N$ player game, with $\varepsilon_N \to 0$.  Then  any
    sub-sequence of the sequence $(\pi^N)$ has a sub-sequence that
    converges weakly to a mean field equilibrium of ${\cal G}$.
  \end{itemize}
\end{theorem}
\begin{proof}
  First,  $V^N(\pi^n,\pi)$ converges to
  $V(\pi^n,\pi)$ uniformly in $(\pi^n,\pi)$.
Uniform convergence follows from Theorem~3.3.2 in
\cite{tembine2009mean} (The theorem is stated for stationary strategies, but  local strategies as defined here are
equivalent to stationary strategies, as defined in
\cite{tembine2009mean}).

  Thus, for any $\varepsilon$, there exists $N_0$ such that $N\ge N_0$
  implies that\\
  $\abs{V^N(\pi^n,\pi)-V(\pi^n,\pi)}\le\varepsilon/2$. Hence, if $\pi$
  is a mean field equilibrium, this implies that for any local strategy
  $\pi^n$:
  \begin{equation*}
    V^N(\pi,\pi) \le V(\pi,\pi)+\frac{\varepsilon}{2} %
    \le V(\pi^n,\pi)+\frac{\varepsilon}{2} \le V^N(\pi^n,\pi)+\varepsilon. 
  \end{equation*}
  This shows \emph{(i)}. 

  For \emph{(ii)}, if $\pi^N$ is a sequence of local strategies, then
  any sub-sequence has a sub-sequence that converge weakly to some
  local strategy $\pi^\infty$. As $V(\pi^n,\pi)$ is continuous in
  $\pi^n$ and $\pi$ (for the weak topology), this implies that
  $V(\pi^\infty,\pi^\infty)\le V(\pi^n,\pi^\infty)$ for all local strategy
  $\pi^n$.
\end{proof}

\subsection{Markov Equilibria May Not Converge to Mean Field Equilibria}
\label{sec:ce}

We now show that Theorem \ref{th:coonv1}-({\it ii}) does not generalize to Markov strategies.
the following example was first presented in \cite{DGG16}.
The main ingredient used to construct the following counterexample, is the
``tit-for-tat'' principle. This principle  can be used to construct equilibria
for any $N$-player game but  cannot be used in mean field games.
This approach has been used in repeated game papers (see for example the examples in \cite{Sabourian}, further generalized \cite{Nabil}). Up to our knowledege, this type of behavior has not yet been described in the mean field game framework.

Let us consider a mean field version of the classical prisoner's dilemma.
The state space of a player is $\calE=\{C,D\}$ (that stand for
Cooperate and Defect) and the action set is the same $\calA=\calE$.
At each time step, one player is chosen. If she selects an action
$a\in\calA$, her state becomes $a$ at the next time step.

The instantaneous cost of a Player~$n$ depends on her state $i$ and on
the mean field $m$:
\begin{equation*}
  c_{i,a}(m) = \left\{
    \begin{array}{ll}
      m_C + 3m_D &\text{ if $i=C$}\\
      2m_D & \text{ if $i=D$}
    \end{array}
  \right.
\end{equation*}
At each time step, this cost function corresponds to a matching game where a player plays
 against a randomly assigned opponent and suffers a cost
that corresponds to the following matrix:
\begin{center}
  \begin{tabular}{|c|c|c|}
    \hline
    &   C   &  D \\\hline
    C&  1,1 &  3,0\\\hline
    D&  0,3 &  2,2\\\hline
  \end{tabular}
\end{center}

The strategy $D$ dominates the strategy $C$. This implies that playing
$D$ is the unique mean field equilibrium. Indeed, the expected cost (given by
\eqref{eq:V(pi0,pi)}) of a Player~0 that has a state vector $x$
while the mean field is $m(t)$ is
\begin{align*}
  & \int_{0}^\infty[x_C(t)(m_C(t) + 3m_D(t))(\pi^0_{CC} (t)+\pi^0_{CD}(t))  + x_D(t)2m_D(t)(\pi^0_{DC}(t) +\pi^0_{DD}(t))]e^{-\beta t}dt\\
  &=\int_{0}^\infty[x_C(t) + 2m_D(t)]e^{-\beta t}dt,
\end{align*} 
by using the fact that $\pi^0_{CC}(t) +\pi^0_{CD}(t) = \pi^0_{DC}(t) +\pi^0_{DD}(t) = 1$ and 
$x_C(t)+x_D(t)=m_C(t)+m_D(t)=1$.

It should be clear that this cost is minimized when $x_C$ is minimal,
which occurs when the strategy is to choose action $D$ regardless of
the current state. This shows that the only mean field equilibrium is
when all players choose action $D$.

Let us now consider the game with $N$ players and consider the
following Markov strategy:
\begin{equation*}
  \pi^N(m)=\left\{
    \begin{array}{ll}
      C& \text{ if $m_C=1$}\\
      D& \text{ if $m_C<1$}
    \end{array}
  \right.
\end{equation*}
and let us show that for $\beta<1$ and $N$ large, $\pi^N$ is a Markov Nash
equilibrium.

Assume that all players, except Player~$n$, play the strategy $\pi^N$
and let us compute the best response of Player~$n$. It should be clear
that if at time $0$, $m_C<1$, then the best response of Player~$n$ is
to play $D$. On the other hand, if $m_C=1$, then:
\begin{itemize}
\item If Player~$n$ applies $\pi^N$, she will suffer a cost
   $\int \exp(-\beta t)dt=1/\beta$.
\item If Player~$n$ deviates from $\pi^N$  and chooses
  the action $D$, all players will also deviate after that time. This implies that $m_D(t) \approx 1-\exp(-t)$ and that the player $n$
  will suffer a cost approximately  equal to
  $\int_0^\infty (x_C(t)+2-2e^{-t})e^{-\beta t}dt \ge
  2/(\beta(\beta+1))$ when $N$ is large.
\end{itemize}
When $\beta<1$, then  $2/(\beta(\beta+1)) > 1/\beta$, so
that Player~$n$ has no incentive to deviate from the strategy $\pi^N$
and that therefore, $\pi^N$ is a Nash equilibrium.  We also observe
that for this example, the value of the finite game does not converge
to the one of the mean field game.

In conclusion to this section, one can argue that this counter-example
should not be surprising because, in mean field games, punishment is
possible against a fraction on the population that deviates but is not
possible against individual deviation, because it is not seen in the
population distribution.

As a final remark, as in the case of repeated games, the continuity  with respect to $m$ (not true here) is critical for convergence (see \cite{Sabourian}).

%
%

\section{Finite Horizon Case}

Let us now consider mean field games over a finite time horizon $T$.
These games are similar to games with discounted costs, previously
defined, but they only run for a finite duration $T$.  As in the
discounted case, the evolution over time of the population
distribution $\bbm^{\pi}$ is given by \eqref{eq:emp-measure} and the
evolution of Player~0's distribution is given by
\eqref{eq:proba-state}.

Given the population strategy $\pi$ and Player~0 strategy $\pi^0$, the
expected cost of Player~0 for the finite horizon case is defined as
follows:
\begin{equation}
  V(\pi^0,\pi)=\int_0^T \left(\sum_{i\in\calE}\sum_{a\in\calA}{x_i(t) c_{i,a}(\bbm^{\pi}(t)) \pi^0_{i,a}(t)}\right)\ dt.
  \label{eq:finite_T}
\end{equation}
In the literature, similar models have been studied, considering
continuous time finite state space mean field games with finite
horizon. The authors in \cite{GMS13} consider uniformly convex cost
functions and in \cite{Gueant2014} cost functions are assumed to be
strictly convex. In our model, we assume that the costs are continuous
in the population distribution.  It can also be observed that the
instantaneous cost of Player~0 is linear in $\pi^0$. Therefore, the
model we study in this work is not covered by these papers.

We define the notion of mean field equilibrium for the finite horizon
case as in the discounted case, by replacing the cost function
\eqref{eq:V(pi0,pi)} by \eqref{eq:finite_T}.  Then, the proof of the
existence Theorem~\ref{thm:mfe-existence} applies \emph{mutatis
  mutandis} to show the existence of a mean field equilibrium in this
case: Any continuous time mean field game over a finite horizon that
satisfies Assumption~(A1) has a mean field equilibrium.

\subsection{Convergence to a Mean Field Equilibrium}

The construction of a counter example of convergence with an infinite
time horizon given in \textsection \ref{sec:ce}  cannot be directly adapted to the finite horizon case.
In the finite-horizon version of  the game defined  in  \textsection  \ref{sec:ce}, the strategy $\pi^N$ is not a Nash
equilibrium for the $N$-player game  because at the last time-slot, the
best response of Player~$n$ to any strategy is to play $D$. By induction
on the number of time-slots, the only Nash equilibrium of the $N$-player game is when all
players play $D$, which coincides with the mean field equilibrium.

Yet, a counter-example also exists for finite-time horizon. 
The essential idea is to start with a matrix game with two pure Nash equilibria instead of one as in the previous example.
 Let us
consider the following cost matrix:
\begin{center}
  \begin{tabular}{|c|c|c|c|}
    \hline
    &   C   &  D   & P \\\hline
    C&  1,1 &  3,0 & 4,0\\\hline
    D&  0,3 &  2,2 & 4,3\\\hline
    P&  0,4 &  3,4 & 3,3\\\hline
  \end{tabular}
\end{center}
The setting is similar to the previous example: the action set is
equal to the state state $\calE=\calA=\{C,D,P\}$ and at each time
step, one player is chosen. If she selects an action $a\in\calA$, then
her state becomes $a$ at the next time step. 
This game can be viewed as a generalization of the prisoner's dilemma
with an additional Nash-equilibrium $P$ (which stands for
``punish''). It can be shown using a similar path as in the previous
section that, when $T$ is large enough, the following time-dependent
Markovian\footnote{When the time horizon is finite, it is natural to
  consider Markovian strategies that depend on time.} strategy is a
Nash equilibrium:
\begin{align}
  \label{eq:pi_finite}
  \pi^N(m,t) = \left\{
  \begin{array}{ll}
    C & \text{ if $t<1$ and $m_c=1$;}\\
    D & \text{ if $t\ge1$ and $m_P=0$;}\\
    P & \text{ otherwise.}
  \end{array}
  \right. 
\end{align}
In the above strategy, the state $P$ is used as a stick to punish
people from deviating from the imposed strategy. In this case, nobody
has an incentive to deviate from this strategy at the last step
because $D$ is also a Nash equilibrium.

The mean field game has only two equilibria: The whole population
always plays $D$, or the whole population always plays $P$. These
equilibria are also equilibria for the finite-game. Yet, they both
have a larger cost than the strategy of
Equation~\eqref{eq:pi_finite}. This leads us to say that the value of
the game does not converge: the asymptotic cost of the
strategy~\eqref{eq:pi_finite} is strictly smaller than the cost of any
of the mean field equilibria.


\section{Synchronous Players}
\label{sec:discrete-time}

As explained in the previous section, mean field games in continuous
time appear naturally as the limit of $N$-player asynchronous games as
$N$ goes to infinity.  In these asynchronous games with $N$ players,
only one player changes state at the same time. However, there are
other situations where it is more natural to consider synchronous
games in which, at each time step, all players take an action.

\subsection{Synchronous $N$-Player Games with Exchangeable Players}

Here we consider a finite synchronous game
${\cal G}_s^{N} = (N,\calE,\calA,\{P_a\},\mathbf{M}_0,\{c_a\},\beta)$
with $N$ identical players with several differences from the model
used in Section~\ref{sec:continuous_N}, the main one being the
replacement of the rate matrices by stochastic matrices.  As before,
each Player~$n$ has an internal state $X_n(t)$ that belongs to a
finite state space $\calE$
($\mathbf{X}(t) = (X_0(t),\ldots,X_{N-1}(t)$) and chooses an action
from a finite action space $\calA$. The main difference with the
previous asynchronous model is that at each time step $t\in\Z^+$, all
players choose an action $A_n(t)\in\calA$ simultaneously.  We assume
that, a player in state $i$ who chooses action $a$ goes to state $j$
with probability $P_{ija}(\bX(t))$ and that, given $\bX(t)$, the
evolution of all players are independent.  Furthermore, we assume that
the players are {\it exchangeable}, {\it i.e.} for any permutation
$\sigma$ of the $N$ players,
$P_{ija}(X_0(t),\ldots,X_{N-1}(t)) =
P_{ija}(X_{\sigma(0)}(t),\ldots,X_{\sigma(N-1)}(t))$. The fact that
all players are exchangeable implies that the dependence in
$\mathbf{X}(t)$ can be replaced by a dependence on the population
distribution $\mathbf{M}(t)$.  More precisely, for any vector state
$\mathbf{x},\mathbf{y}\in\calE^N$ and any action vector
$\mathbf{a}\in\calA^N$, one can write:
\begin{equation}
  \proba{\mathbf{X}(t+1)=\mathbf{j} \middle| \mathbf{X}(t)=\mathbf{i} , \mathbf{A}(t)=\mathbf{a} ,\calF_t} = \prod_{n=1}^N P_{i_nj_na_n}(\mathbf{M}(t)),
  \label{eq:R_N}
\end{equation}
where  $\calF_t$ is the natural filtration of the game up to time $t$,
$\mathbf{m}$ is  the population distribution   of $\mathbf{x}$ and  $\forall i,j  \in \calE, \forall a \in \calA, P_{ija}(\mathbf{m})$ forms a stochastic matrix, continuous in $\mathbf{m}$.

The  instantaneous cost at time $t$  depends on actions
and state at time $t-1$,  symmetric in
all players, so it can be written as a function of the population distribution: $c_{X_n(t),A_n(t)}(\mathbf{M}(t))$,
and a discount factor $\delta$ at each time step.  Given a
strategy $\pi^0$ used by Player~0 and a strategy $\pi$ used by all
the others, the expected cost of Player~$n$ is:
\begin{equation}
  V^N(\pi^0,\pi) = \expect{(1-\delta) \sum_{t=0}^\infty \delta^tc_{X_0(t),A_0(t)}(\mathbf{M}^{\pi}(t))\middle| %
    \begin{array}{l}
      \text{$A_0$ is chosen w.r.t.  $\pi^0$}\\
      \text{$A_{n'}$ is chosen w.r.t. $\pi$ if $n'\ne 0$}
    \end{array}
  }.
  \label{eq:V_N_discret}
\end{equation}

\subsection{Corresponding Mean field Game}
\label{sec:discrete_time_mean_field_discounted}

Synchronous games also admit mean field game limits.
To construct this limit, let us consider a strategy $\pi$ such that $\pi_{i,a}(\mathbf{m})$ is the
probability for a player to choose action $a$ given that she is in
state $i$ and that $\mathbf{M}(t)=\mathbf{m}$. Assume that $\mathbf{M}(0)$ converges in
probability to some $\mathbf{m}(0)$ as $N$ goes to infinity and that all players
except Player~0 apply a strategy $\pi$ that is continuous in $\mathbf{m}$. As
shown in Theorem 1 in \cite{gast2011mean} (up to differences in notations,
the mean field model in \cite{gast2011mean} is the same as Equation \eqref{eq:R_N}), the population distribution $\mathbf{M}^\pi(t)$
converges (in probability) to a deterministic quantity $\mathbf{m}^\pi(t)$ as
$N$ goes to infinity. $\mathbf{m}^\pi(t)$ is defined by
\begin{equation}
  m^{\pi}_j(t+1)=\sum_{i\in\calE}\sum_{a\in\calA}{m^{\pi}_i(t) P_{i,j,a}(\mathbf{m}^{\pi}(t)) \pi_{i,a}(\mathbf{m}(t))}.
  \label{eq:emp-measure-discrete}
\end{equation}
We denote by ${\pi}^0$ the strategy of Player~0. The
probability that Player~0 is in state $j\in\calE$ evolves over time
according to the following equation:
\begin{equation}
x_j(t+1)=\sum_{i\in\calE}\sum_{a\in\calA}{x_i(t) P_{i,j,a}(\mathbf{m}^{\pi}(t)) \pi^0_{i,a}(\mathbf{m}(t))}.
\label{eq:proba-state-discrete}
\end{equation}

In this case, the cost of Player~0, given by \eqref{eq:V_N_discret} becomes
\begin{align*}
  V({\pi}^0,\pi)=(1-\delta) \sum_{t=0}^{\infty}\sum_{i\in\calE}\sum_{a\in\calA}{ \delta^{t}x_i(t) c_{i,a}(\mathbf{m}^{\pi}(t)) \pi^0_{i,a}(\mathbf{m}(t))}.
\end{align*}

As the evolution of $m$ is deterministic, for any closed loop strategy
$\pi_{i,a}(\mathbf{m}(t))$ and any initial condition $\mathbf{m}(0)$,
there exists an open-loop strategy $\pi_{i,a}(t)$ that leads to the
same values for $\mathbf{m}^{\pi}(t)$ and the same cost. Hence, for
the mean field model, one can replace any state-dependent strategy
$\pi(\mathbf{m}(t))$ in the above equations by a time-dependent
strategy $\pi(t)$.

Player~0 chooses the strategy that minimizes her expected cost. When Player~0 does so, we say
it uses the best response to the mass strategy $\pi$.
\begin{equation*}
  BR(\pi)=\argmin_{\pi^0}V(\pi^0,\pi).
\end{equation*}

A strategy is said to be a mean field equilibrium if it is a fixed point
for the best response function, that is,
\begin{equation*}
  \pi^{MFE}\in BR(\pi^{MFE}).
\end{equation*}

One of the difficulties of the analysis of continuous time mean field
game is that the elements under consideration (the population distribution,
the population strategy, Player~0 strategy...)  are continuous functions
of time. In the discrete time case, the model gets significantly
simplified since all the elements are vectors. Hence, the proof of the
existence of a mean field equilibrium for continuous-time mean field
game (Theorem~\ref{thm:mfe-existence}) can be adapted to show that the
following result.

\begin{theorem}[Mean Field Equilibrium Existence for Synchronous Games]
Any synchronous  mean field game with discounted cost that satisfies
Assumption (A1) for $P$  and $c$ respectively, has a mean field equilibrium.
\label{thm:mfe-existence-discrete-discounted}
\end{theorem}

\begin{proof}[Sketch of proof]
  We first observe that the set of discrete-time open-loop policies is
  a compact and convex set. Thus, to finish the proof, we need to show
  that the best response function has a closed graph and it is
  convex. The former condition is true since the set of open-loop
  policies belongs to a finite dimensional space and from the
  continuity assumptions (A1). The last condition can be shown using
  the same arguments as in the proof of
  Theorem~\ref{thm:mfe-existence}.
\end{proof}

\subsubsection{An Important Special Case: Repeated Games}
\label{sec:repeat}

The classical repeated games with discounted costs and with identical
players form a subclass of synchronous games, as defined here.
To see this, let us first consider a static  $N$-player matrix game $G$ with
symmetric cost: $u(a_1,\ldots,a_N)$ is the instantaneous cost of any player when
the players use actions $a_1,\ldots,a_N$ respectively.
Furthermore, we assume that $u(a_1,\ldots,a_N) =
u(a_{\sigma_1},\ldots,a_{\sigma_N})$, for any permutation $\sigma$ of
$\{1,\dots,N\}$.
The players repeat the matrix game infinitely often and their
cost under strategy $\pi^1,\cdots,\pi^N$ is the
discounted sum of the costs:
\begin{equation}
  V^N(\pi^1,\pi^2,\cdots,\pi^N) = (1-\delta)\sum_{t=0}^\infty \delta^{ t}
  u(\pi^1(t),\pi^2(t),\cdots,\pi^N(t)).
\label{eq:repeatCost}
\end{equation}

These games fit in our framework: The state of a player is merely her
current action ($\mathbf{X}(t) = \mathbf{A}(t)$) and the evolution of
the state becomes trivial: Under state $x=a$ and selecting action $b$,
the next state does not depend on the other players and  becomes $b$
with probability  one:  $P_{ab}(b,\mathbf{M}(t)) = 1$.
The cost of one player at each stage corresponds to an
immediate cost  $c_{X_n(t),A_n(t)}(\mathbf{M}(t)) = u
(\mathbf{X}(t))$ since the cost $u$ only depends on the population
distribution by symmetry.
As for the total cost of a player,  
\eqref{eq:repeatCost} coincides with \eqref{eq:V_N_discret}, as long
as all players in the same state use the same strategy.

\subsection{The Folk Theorem Does Not Scale}

The relation between equilibria of $N$-player games with their mean
field limits is also complex in the discrete time case.

Let us first focus on results that concern the performance of mean
field equilibria in the $N$-player game.  The situation is almost
similar to the continuous time case and resembles Theorem
\ref{th:coonv1} (i) in the sense that if $\pi$ is a mean field
equilibrium, then under assumption (A1), there exists $N_0$ such that
for all $N\ge N_0$, $\pi$ is a local $\varepsilon$-equilibrium of the
$N$-player game.  The proof of this is essentially similar to the
proof of Theorem \ref{th:coonv1}.

Let us now consider the Nash equilibria of the $N$-player game.
The situation is very different from the continuous time case because
the state of all the players can change in one time unit in the discrete
time  while in 
continuous time, 
state can only change in small steps, one player at a time.

This has several consequences on the nature of equilibria under both 
models.
As mentioned before,  the Nash equilibria in the continuous time case
may depend on the initial population distribution, but this is not the
case here, so that there is more latitude for designing equilibria.

Let us consider the particular case of repeated games, introduced in
Section \ref{sec:repeat}.
For this type of games, the set of equilibria can be characterized 
using the Folk Theorem for repeated games.

\begin{theorem}[Folk theorem, adapted from Theorem~A in
  \cite{fudenberg1986folk})] Let $G$ be a symmetric matrix game,
  and let $V^*$ be the cost under the strategy that repeats the Nash
  equilibrium of the static game $G$.  Then for any
  compatible\footnote{In this context, a compatible cost is a cost
    that can be attained by at least one strategy. } cost $V$ smaller
  than $V^*$, there exists a discount factor $\delta\in(0,1)$ such
  that $V^*$ is the cost of an equilibrium of the discounted repeated
  game.
\end{theorem}
 
Actually, for any $V < V^*$, the construction of an equilibrium whose
cost is $V$ is based on the ``tit for tat'' principle.  We claim that
none of these equilibria scale at the mean field limit.  Let us
consider the following example for a static game.  Each player only
has two strategies, $D$ and $C$.  If all players play $D$, the cost is
$-1$.  If all players play $C$, the cost is $-2$.  If some players
play $D$ and others play $C$, then, all the players who play $C$ get
$-2 M_C$ while the players who play $D$ get $-3 M_C - M_D$, where
$M_C$ and $M_D$ are the proportions of players playing $C$ and $D$
respectively. These costs correspond to the average costs obtained by
a player in a matching game against a random opponent.

The unique Nash equilibrium of the static game is strategy
$(D,D,\ldots, D)$.
The cost of the corresponding repeated game is 
$(1-\delta) \sum_t -\delta^t = -1$.

Let us now consider the following strategy (denoted $\pi^N$ in the
following) for all players: Play $D$ for $k$ rounds then play $C$ as
long as every-other player has followed the same pattern, else play
$D$ forever.  The cost of this strategy is between $-1$ and $-2$:
\[(1-\delta) (\sum_{t=0}^{k-1}  -\delta^t + \sum_{t=k}^\infty -2
\delta^t) =  -1-\delta^k.\]

The strategy $\pi^N$ is an equilibrium of the finite game if $\delta$
is large enough.  Indeed, no player wants to deviate in the first $k$
rounds, because her cost would increase: In the rounds after $k$, a
deviation provides an immediate cost advantage, at the cost of being
punished until the end of time, so that a larger enough $\delta$ makes
this non-profitable.

Let us now consider the mean field game setting. If the whole
population uses the strategy $\pi^N$ and if Player~0 uses the same
strategy her cost becomes
\begin{align*}
  V(\pi^N,\pi^N) &=(1-\delta)
                   \sum_{t=0}^{\infty}\sum_{i\in\calE}\sum_{a\in\calA}{
                   \delta^{t}x_i(t)
                   c_{i,a}(\mathbf{m}^{\pi}(t)
                   \pi^0_{i,a}(\mathbf{m}(t))}, \\
 &=(1-\delta) (\sum_{t=0}^{k-1}  -\delta^{t} + \sum_{t=k}^{\infty} -2
   \delta^{t}) \\
 &=  -1-\delta^k.
\end{align*}
However in the mean field setting, the best response of Player~0 to
$\pi^N$ is not $\pi^N$ but the strategy $\pi^D$ where she plays $D$
all the time.  Indeed in this case her total cost becomes
\begin{align*}
  V(\pi^D,\pi^N)
 &=(1-\delta) (\sum_{t=0}^{k-1}  -\delta^{t} + \sum_{t=k}^{\infty} -3
   \delta^{t}) \\
 &=  -1- 2 \delta^k.
\end{align*}
This shows that $\pi^N$ is not a mean field equilibrium and a ``free
rider'' player can take advantage of the fact that the population will
not act against her.

\subsection{Finite Horizon Case}

We now focus on the mean field games when objects evolve in discrete time 
time over a finite horizon,  $0$ to $T$. In this case, the population distribution $\mathbf{m}^{\pi}$ is defined 
by \eqref{eq:emp-measure-discrete}, which depends on the strategy of the 
mass $\pi$. We assume that Player~0 can choose her own strategy $\pi^0$. The
expected cost of Player~0 is
\begin{align*}
  V({\pi}^0,\pi)=\sum_{t=0}^{T}\sum_{i\in\calE}\sum_{a\in\calA}{x_i(t) c_{i,a}(\mathbf{m}^{\pi}(t)) \pi^0_{i,a}(\mathbf{m}(t))},
\end{align*}

\noindent where $x_i(t)$ is the probability that Player~0 is in state $i$ at time $t$. 
The evolution of $x_i(t)$ over time is described in \eqref{eq:proba-state-discrete}.

Player~0 uses best response to a given population strategy $\pi$,
which means that she selects the strategy $\pi^0$ that minimizes her
expected cost. We are interested in proving the existence of a mean
field equilibrium which consists of finding a strategy that is a
fixed-point for the best response function. In
Section~\ref{sec:discrete_time_mean_field_discounted}, we showed this
for the discounted case. In the finite horizon case, the vectors have
finite size and, as a consequence, it is immediate to show, using the
same arguments of those required for the proof of
Theorem~\ref{thm:mfe-existence-discrete-discounted}, that any discrete
time mean field game with finite horizon cost such that $P$ and $c$
satisfy Assumption $(A1)$ has a mean field equilibrium.
Again, the proof mimics the proof of the analog Theorem
\ref{thm:mfe-existence} in continuous time over a finite horizon.


%
%


\section{Conclusions}
\label{sec:conclusions}

In this article, we generalize the framework of discrete-space mean
field games to the cases of non-convex costs and explicit
interactions. They hit a good compromise between tractability
(existence of an equilibria) and modelization power (including
propagation and congestion behaviors). This model consists of a finite
state space mean field game where the transition rates of the objects
and the cost function of a generic object depend not only on the
actions taken but also on the population distribution. We also show
that there exists a sub-class of Nash equilibria for $N$-player games
that converge to mean field equilibria when the number of players goes
to infinity.  Outside of this class, and in particular for all
equilibria using the ``tit for tat'' principle over which the Folk
theorem is based, the convergence does not hold.

For future work, we are interested in finding conditions ensuring the
uniqueness of the mean field equilibrium. We believe that monotony
assumptions similar to assumptions in \cite{GMS13} are required to
prove the existence of a unique mean field equilibrium in this model.
On the other hand, another interesting open question concerns the
convergence of $N$-players equilibria to mean field equilibria when
the number of player grows. We believe that there exist many
$N$-player games for which the only limiting equilibria are mean field
equilibrium, for example when players have incomplete information about the game. It would  be interesting to characterize the
sub-class of strategies where  convergence to mean field equilibria
holds. Obviously, this class includes all local
strategies (no information) and excludes some Markovian ones (full information).

\bibliographystyle{abbrv}
\bibliography{bibfile}

\end{document}